\documentclass[a4paper, 10pt]{amsart}
\usepackage{amsmath,amssymb,amsthm,graphicx, amscd, pstricks, hhline,
enumerate, xcolor, pstricks-add, auto-pst-pdf}
\usepackage{listings, enumitem, tikz}
\usetikzlibrary{matrix,arrows}
\usepackage[latin1]{inputenc}

\DeclareMathOperator{\Sl}{Sl}

\DeclareMathOperator{\Hom}{Hom}

\DeclareMathOperator{\Heis}{Heis}

\DeclareMathOperator{\res}{res}
\DeclareMathOperator{\uf}{uf}
\DeclareMathOperator{\hd}{hd}
\DeclareMathOperator{\Ima}{Im}

  \newcommand{\N}{\ensuremath{\mathbb{N}}}
  \newcommand{\Z}{\ensuremath{\mathbb{Z}}}
  
  \newcommand{\R}{\ensuremath{\mathbb{R}}}

 \newcommand{\fa}[1]{\forall_{#1}\;\;\;}
 \newcommand{\exi}[1]{\exists_{#1}\;\;\;}
 
 \newcommand*{\longhookrightarrow}{\ensuremath{\lhook\joinrel\relbar\joinrel\rightarrow}}

\setenumerate[1]{label = (\roman*)}
\begin{document}
\title{Uniformly finite homology and amenable groups}
\author{Matthias Blank
}
 \email{
  \texttt{matthias.blank@mathematik.uni-regensburg.de}
 }
 \address{
Fakult\"at f\"ur Mathematik,
Universit\"at Regensburg,
93040 Regensburg,
Germany
 }
 \author{Francesca Diana
 }
 \email{
   \texttt{francesca.diana@mathematik.uni-regensburg.de}
 }
 \address{
Fakult\"at f\"ur Mathematik,
Universit\"at Regensburg,
93040 Regensburg,
Germany
}
\address{
Graduiertenkolleg ``Curvature, Cycles and Cohomology''
 }
\ 
\date{\today}
 \begin{abstract}
Uniformly finite homology  is a coarse invariant for metric spaces; in particular, it is a quasi-isometry invariant for finitely generated groups. In this article, we study uniformly finite homology of finitely generated amenable groups and prove that it is infinite dimensional in many cases. The main idea is to use different transfer maps to distinguish between classes in uniformly finite homology. Furthermore we show that there are infinitely many classes in degree zero that cannot be detected by means.

\end{abstract}
\maketitle

\renewcommand{\labelenumi}{(\roman{enumi})}
\newtheorem{Satz}{Satz}[section]
\newtheorem{thm}[Satz]{Theorem}
\newtheorem*{theorem}{Theorem}
\newtheorem{prop}[Satz]{Proposition}
\newtheorem{cor}[Satz]{Corollary}
\newtheorem*{Ziel}{Ziel}
\newenvironment{bew}{\begin{proof}[Beweis:]}{\end{proof}}
\newtheorem*{hilfs}{Hilfsbehauptung}
\theoremstyle{definition}
\newtheorem{bsp}[Satz]{Beispiel}
\newtheorem{exa}[Satz]{Example}
\newtheorem{algo}[Satz]{Algorithmus}
\newtheorem{exc}{"Ubungsaufgabe}
\newtheorem{def.}[Satz]{Definition}
\newtheorem*{beh}{Behauptung}
\newtheorem*{bem}{Bemerkung}
\newtheorem*{mot}{Motivation}
\newtheorem*{erin}{Erinnerung}
\newtheorem*{rec}{Recall}
\newtheorem*{as}{Assumption}
\newtheorem{rem}[Satz]{Remark}
\newtheorem{lemma}[Satz]{Lemma}
\lstset{numbers = left, numberstyle = \tiny, numbersep = 5pt}
\lstset{language=C}

\section{Introduction}

Uniformly finite homology was introduced by Block and Weinberger \cite{BW} to study the large-scale structure of metric spaces having bounded geometric complexity. It is a coarse invariant in the sense that two quasi-isometric metric spaces have isomorphic uniformly finite homology. The chains considered are the one introduced by Roe \cite{R} to define coarse homology, but with an additional boundedness condition on the coefficients (Definition~\ref{d:uf}). 

The uniformly finite homology groups of a metric space $X$ are denoted by $H_*^{\uf}(X;\R)$. One of the main properties of $H_0^{\uf}(X;\R)$ is the relation with the notion of amenability for spaces having coarse bounded geometry (in particular for finitely generated groups). The following result is due to Block and Weinberger~\cite[Theorem 3.1]{BW}:
\begin{thm}
If $X$ is a metric space with coarse bounded geometry, then $X$ is amenable if and only if $H_0^{\uf}(X;\R)\neq 0$.
\end{thm}

The zero degree uniformly finite homology group has been used for a number of different applications. For instance, rigidity problems for uniformly discrete metric spaces with bounded geometry have been investigated with the help of uniformly finite homology \cite{D,W}. Moreover vanishing classes in $H_0^{\uf}(X;\R)$ have been used to construct aperiodic tilings for non-amenable manifolds \cite{BW}.

So most of the applications of uniformly finite homology use the fact that $H_0^{\uf}(X;\R)$ vanishes for a non-amenable metric space with bounded geometry. On the other hand, it is not so clear how $H_0^{\uf}(X;\R)$ looks like in the amenable case. Moreover, higher degree uniformly finite homology groups are not yet well understood \cite{BW1, DR3}. In the case that $G$ is a finitely generated group there is a ``lifting'' map $H_*(G;\R)\longrightarrow H_*^{\uf}(G;\R)$, which is injective if $G$ is amenable \cite{ABW, AT}. Dranishnikov~\cite{DR1, DR2} uses this map to study large-scale notions of dimension.

In this paper, we will show that uniformly finite homology groups of amenable groups are infinite dimensional in many cases. Our main result is:

\setcounter{section}{3}
\setcounter{Satz}{7}
\begin{thm}
Let $n\in\N$ and let $G$ be a finitely generated amenable group. Let $H\leq G$ be a subgroup such that $[G:H]=\infty$ and that the map
\begin{align*}
 H_n(i)\colon H_n(H;\R)\longrightarrow H_n(G;\R)
\end{align*}
induced by the inclusion $i:H\longhookrightarrow G$ is non-trivial. Then $\dim_\R(H_n^{\uf}(G;\R)) = \infty$. 
\end{thm}
\setcounter{section}{4}
\setcounter{Satz}{3}

Therefore, for many amenable groups we can use regular group homology to deduce that uniformly finite homology is infinite dimensional.

We present some examples following directly from this approach. For instance, from the calculation of group homology of nilpotent groups we can deduce:
\begin{exa}[Nilpotent groups]
Let $G$ be a finitely generated virtually nilpotent group of Hirsch rank $h\in\N$. Then
\begin{align*} 
 H_k^{\uf}(G;\R) = 
 \begin{cases}
  \text{infinite dimensional} &\text{ if  $k \in \{0,\dots, h-1\}$}\\
  \R &\text{ if $k=h$}\\
 0 &\text{ else.}
 \end{cases}
\end{align*}
\end{exa}
\setcounter{section}{5}
\setcounter{Satz}{0}
Our main technique for proving Theorem~\ref{not-normal} is to use invariant means on $G$ and the corresponding transfer maps to distinguish between different homology classes. This is straightforward in the zero degree case and doing so we give a very simple proof that infinite amenable groups $G$ have infinite dimensional $H_0^{\uf}(G;\R)$. In order to extend this to the higher degree case, we construct functions in $\ell^{\infty}(G)$ that are invariant with respect to a subgroup $H\leq G$ and can be distinguished by means.\\ 

We also define a subspace $\widehat{H}_0^{\uf}(G;\R)\subset H_0^{\uf}(G;\R)$ of classes that can \emph{not} be detected by means and hence are different from the classes constructed before.  In the last part of this article, we give a geometric condition for classes in uniformly finite homology to be in $\widehat{H}_0^{\uf}(G;\R)$. We then present a method of differentiating between such classes, following Whyte. Hence we can show:
\begin{thm}
 Let $G$ be a finitely generated infinite amenable group. Then \[\dim_{\R} \widehat{H}_0^{\uf}(G;\R) =\infty.\] 
\end{thm}
\setcounter{section}{1}
\bigskip

The article is structured in the following way: In Section~\ref{notation} we introduce  uniformly finite homology both following Block and Weinberger and in terms of group homology with $\ell^{\infty}$-coefficients and discuss its properties most relevant for this article. Section~\ref{main} contains our main results and proofs. In Section~\ref{examples} we present some examples and immediate applications of our main theorem. Finally, in Section~\ref{sparse} we introduce and study \emph{sparse classes} in $H_0^{\uf}(G;\R)$ and prove Theorem~\ref{t:sparse}.

\bigskip

\paragraph*{\bf{Acknowledgments}}
We thank Cristina Pagliantini for numerous fruitful discussions. We are grateful to Piotr Nowak for coming to Regensburg and suggesting the question. Many thanks to Cristina Pagliantini and Malte R\"oer for their support. Finally, we especially like to thank Clara L\"oh for many helpful suggestions and discussions. 

\section{Basic facts and notation}\label{notation}

In this section we fix some notation and present the main object we will investigate, namely uniformly finite homology. This was introduced by Block and Weinberger as a coarse homology invariant. The basic references for uniformly finite homology are \cite{BW, NY, W}. We will also define group homology with $\ell^{\infty}$-coefficients. For finitely generated groups there is a canonical isomorphism between the corresponding chain complexes (Remark~\ref{uf-infty}). Hence, to study uniformly finite homology we will often use the description in terms of $\ell^{\infty}$-coefficients.

\begin{def.}\label{d:uf}
Let $(X,d)$ be a metric space.
\begin{enumerate}
 \item   For each $n\in\N$ denote by $C_n^{\uf}(X;\R)$ the real vector space of functions $c\colon X^{n+1} \longrightarrow \R$ satisfying:
\begin{enumerate}
 \item The map $c$ is bounded. 
 \item For all $r\in\R_{>0}$  there exists a $K_r\in\R_{>0}$ (depending on $c$) such that for all $y\in X^{n+1}$ 
\[
 \bigl|\{x\in B_r(y)\;|\; c(x)\neq 0\}\bigr| \leq K_r. 
\]
Here we consider $X^{n+1}$ endowed with the metric
\[
d_n(x,y):=\max_{i\in\{0,\dots, n\}}d(x_i,y_i).
\]

\item There exists an $R_c\in \R_{>0}$ (depending on $c$) such that 
\[
\fa{x\in X^{n+1}} \sup_{i,j\in\{0,\dots,n\}} d(x_i,x_j)> R_c \Longrightarrow c(x) =0. 
\]
\end{enumerate}
We will write such a function also as a formal sum $\sum_{x\in X^{n+1}} c(x)\cdot x$.
\item Define for each $n\in\N_{>0}$ a boundary operator \[\partial_n\colon C_n^{\uf}(X;\R)\longrightarrow C_{n-1}^{\uf}(X;\R)\] by setting for each $x\in X$
\[
 \partial_n (x) = \sum_{j=0}^n (-1)^{j}(x_0,\dots,\widehat{x}_j,\dots,x_{n}). 
\]
and extending to $C_n^{\uf}(X;\R)$ in the obvious way. In this fashion, we get indeed a chain complex.
\item The homology of $(C_n^{\uf}(X;\R),\partial_n)_{n\in\N}$  is called the \emph{uniformly finite homology of $X$} and denoted by $H_\ast^{\uf}(X;\R)$. 
\end{enumerate}
\end{def.}

An important fact about uniformly finite homology is that it is a coarse invariant~\cite[Proposition 2.1]{BW}: 

\begin{prop}\label{t:Franziska}
 Let $X$ and $Y$ be metric spaces and $f\colon X\longrightarrow Y$ a quasi-isometry. Then $f$ induces a chain map
\begin{align*}
C_n^{\uf}(X;\R)&\longrightarrow C_n^{\uf}(Y;\R)\\
 \sum_{x\in X^{n+1}} c(x)\cdot x &\longmapsto \sum_{x\in X^{n+1}} c(x) \cdot (f(x_0), \dots, f(x_n)). 
\intertext{%
The induced map in homology 
}
H_\ast^{\uf}(X;\R)&\longrightarrow H_\ast^{\uf}(Y;\R)
\end{align*}
is an isomorphism in every degree. 
\end{prop}

In particular, we can define the uniformly finite homology of a group $G$ with a fixed set of generators $S$, by endowing $G$ with the word metric with respect to~$S$. We recall the definition:

\begin{def.}\label{d:wordmetric}
Let $G$ be a group with generating set $S$. The \emph{word metric on $G$ with respect to $S$} is the metric defined as
\[
d_{S}(g,h):=\min\bigl\{n\in\N \; \bigm | \; \exists_{s_{1},\dots,s_{n}\in S\cup S^{-1}}\; g^{-1}\cdot h = s_{1} \cdots s_{n} \bigr\}
\]
for any $g,h\in G$.
\end{def.}

By Proposition~\ref{t:Franziska} the uniformly finite homology of a finitely generated group does not depend on the finite generating set up to canonical isomorphism. Hence, this defines a quasi-isometry invariant for finitely generated groups. Notice that in this case condition (i)-(b) in Definition~\ref{d:uf} can be dropped.

Most important for us and many applications is the following result of Block and Weinberger~\cite[Theorem 3.1]{BW}:

\begin{thm}
Let $G$ be a finitely generated group. Then $G$ is amenable if and only if $H_0^{\uf}(G;\R)\neq 0$.
\end{thm}

We define now group homology with $\ell^{\infty}$-coefficients. Let $G$ be a (discrete) group. Consider the $\R[G]$-chain complex $(C_*(G;\R),\partial_*)$ defined as follows:
\begin{enumerate}
\item For any $n\in\N$, set $C_n(G;\R):=\bigoplus_{(g_0,\dots,g_n)\in G^{n+1}}\R\cdot(g_0,\dots,g_n)$ with the $G$-action given by $g\cdot (g_0,\dots, g_n)=(gg_0,\dots,gg_n)$.
\item For any $n\in \N_{>0}$, let $\partial_n$ be the boundary map
\begin{align*}
\partial_{n}\colon C_n(G;\R) &\longrightarrow C_{n-1}(G;\R)\\
(g_0,\dots,g_n) &\longmapsto \sum_{j=0}^{n}(-1)^{j}(g_0,\dots,\hat{g}_j,\dots,g_n).
\end{align*}
\end{enumerate}
Consider $\ell^{\infty}(G)$ the space of real-valued bounded functions on $G$; this is a left $\R[G]$-module with respect to the action
\begin{align*}
G\times\ell^{\infty}(G) &\longrightarrow \ell^{\infty}(G)\\
(g,\varphi) &\longmapsto g\cdot\varphi =\bigl( g'\longmapsto \varphi(g^{-1}g')\bigr).
\end{align*}
Let $C_*(G;\ell^{\infty}(G))$ be the $\R$-chain complex given by
\[
C_*(G;\ell^{\infty}(G)):=\overline{C}_*(G;\R)\otimes_{\R[G]}\ell^{\infty}(G). 
\]
where $\overline{C}_*(G;\R)$ denotes the right $\R[G]$-module obtained by $C_*(G;\R)$ via the canonical involution $g\mapsto g^{-1}$.
\begin{def.}
For a group $G$ we define the \emph{group homology of $G$ with coefficients in $\ell^{\infty}(G)$} by $H_*(G;\ell^{\infty}(G)):=H_*(C_*(G;\ell^\infty(G)))$.
\end{def.}

The following Remark was observed by Brodzki, Niblo and Wright~\cite{BNW}: 

\begin{rem}\label{uf-infty}
For a finitely generated group $G$, the uniformly finite chain complex $(C^{\uf}_*(G;\R),\partial_*)$ is canonically chain isomorphic to $(C_*(G;\ell^{\infty}(G)),\partial_*)$. 

Indeed, for any $n\in\N$, the simplices in a fixed uniformly finite $n$-chain $c$ are tuples $(g_0,\dots,g_n)\in G^{n+1}$ having diameter less than a uniform constant $R_c \in \R_{>0}$. These simplices are contained in the $G$-orbits of the finitely many simplices of the form $(e,t_1,\dots,t_{n})$ of diameter less than $R_c$. Hence, the following chain isomorphism is well-defined:
\begin{align*}
\rho_n\colon  C_n^{\uf}(G;\R) &\longrightarrow C_n(G; \ell^{\infty}(G))\\
 \sum_{(g_0,\dots,g_n)\in G^{n+1}} c{(g_0,\dots g_n)}\cdot (g_0,\dots, g_n) &\longmapsto \sum_{(t_1,\dots,t_n)\in G^{n}} (e,t_1,\dots, t_n)\otimes \varphi_{(t_1,\dots, t_n)}
\end{align*}
Here for each $(t_1,\dots,t_n)\in G^n$ the map $ \varphi_{(t_1,\dots,t_n)} \in \ell^{\infty}(G)$ is given by
\begin{align*}
\varphi_{(t_1,\dots,t_n)} (g) =  c(g^{-1},g^{-1}\cdot t_1,\dots,g^{-1}\cdot t_n).
\end{align*} 
\end{rem}
In particular, $H_*^{\uf}(G;\R)\cong H_*(G;\ell^\infty(G))$. Therefore, certain aspects of uniformly finite homology are accessible through methods of ordinary group homology.

We recall the definition of amenable groups that will be the main class of groups we will study in this article.

\begin{def.}
\label{definitionmean}
Let $G$ be a group. A \emph{mean} is a linear map $m: \ell^{\infty}(G)\longrightarrow\R$ satisfying the following properties:
\begin{enumerate}
\item We have $m(\chi_G)=1$ for the characteristic function $\chi_{G}\in\ell^{\infty}(G)$ of the group.
\item For any $\varphi\in\ell^{\infty}(G)$ such that $\varphi\geq 0$ (i.e., $\varphi(g)\geq 0$ for all $g\in G$) we have $m(\varphi)\geq 0$.
\end{enumerate}
A mean $m$ is said to be \emph{left $G$-invariant} if
\begin{enumerate}
\setcounter{enumi}{2}
\item For any $g\in G$ and any $\varphi\in\ell^{\infty}(G)$, we have  $m(g\cdot\varphi)=m(\varphi)$. 
\end{enumerate}
\end{def.}

Similarly one could define right $G$-invariant means.

\begin{def.}
A group $G$ is \emph{amenable} if it admits a left $G$-invariant mean.
\end{def.}

\begin{rem}The class of amenable groups contains for example all solvable, all finite and all subexponential groups and is closed under taking extensions, subgroups, quotients, etc. \cite{CC, P}. 
\end{rem}

For an amenable group $G$, let $M(G)$ be the set of left $G$-invariant means and let $LM(G)$ be its linear span in $\Hom_{\R}(\ell^\infty(G),\R)$. 
We have the following result due to Chou~\cite{C}:

\begin{thm}\label{t:Chou}
If $G$ is an infinite amenable group, then $G$ has exactly $2^{{2}^{\left|G\right|}}$ left invariant means, where $\left|G\right|$ denotes the cardinality of $G$. Thus $LM(G)$ is infinite dimensional.
\end{thm}

\section{Main results}\label{main}

In this section, we compute uniformly finite homology of finitely generated (discrete) amenable groups in many cases:  in Theorem~\ref{zerocase} we show that in degree zero this is always infinite dimensional, while in Theorem~\ref{not-normal} we calculate higher degree uniformly finite homology.

\subsection{Transfer via means}
Let $G$ be an amenable group. Every left $G$-invariant mean $m\in M(G)$ induces a transfer map $m_* :C_*(G;\ell^{\infty}(G))\longrightarrow C_*(G;\R)$, which averages the coefficients. Our next proposition is similar to a result of Attie \cite[Proposition 2.15]{AT}.

\begin{prop} \label{thm: mean}
Let $G$ be an amenable group and let $\chi_{G}\in \ell^{\infty}(G)$ be the characteristic function of $G$. Then every mean $m\in M(G)$ induces a transfer map 
\begin{align*}
m_\ast\colon  H_\ast(G;\ell^{\infty}(G))&\longrightarrow H_\ast(G;\R),\\
 [c\otimes \varphi] &\longmapsto [m(\varphi)\cdot c]
\intertext{which is a left inverse to the map}
i_\ast\colon  H_\ast (G;\R)&\longrightarrow H_\ast(G;\ell^{\infty}(G))\\
[c]&\longmapsto [c\otimes\chi_G]
\end{align*}
induced by the canonical inclusion $\R\longhookrightarrow \ell^{\infty}(G)$ as constant functions.
\end{prop}
\begin{proof}
Consider $\R$ as $G$-module with the trivial action; then every $m\in M(G)$ is a $G$-equivariant map $m\colon \ell^{\infty}(G)\longrightarrow\R$ and thus induces a change-of-coefficients chain map 
\begin{align*}
 m_\ast\colon  C_\ast(G;\ell^{\infty}(G))&\longrightarrow C_\ast(G;\R)\\
 c\otimes \varphi &\longmapsto m(\varphi)\cdot c.
\end{align*}
The inclusion $i\colon\R\longrightarrow\ell^\infty(G)$ is also $G$-equivariant and induces a chain map
\begin{align*}
 i_\ast\colon C_\ast(G;\R) &\longrightarrow C_\ast(G;\ell^{\infty}(G))\\
 c&\longmapsto c\otimes\chi_G.
\end{align*}
Obviously, for every mean $m\in M(G)$ the map $m_\ast$ is a left inverse for $i_\ast$. Applying homology proves the claim.\qedhere
\end{proof} 

\begin{rem}\label{t:Mattgias}
Under the canonical identification between uniformly finite homology and homology with $\ell^\infty$-coefficients for finitely generated groups, the maps $m_\ast$ defined in Proposition~\ref{thm: mean} translate (by precomposition with $\rho_\ast$) to transfer maps for uniformly finite homology that we denote by $\overline{m}_\ast$.
\end{rem}
Hence, for any finitely generated amenable group there is an inclusion 
\[
H_\ast(G;\R)\longhookrightarrow H_\ast^{\uf}(G;\R).
\]

\begin{rem}
Given $m\in M(G)$, in degree zero we have $C_0(G;\ell^{\infty}(G))\cong\ell^{\infty}(G)$ and the transfer map $m_0$ coincides with the mean $m$.  Also $C_0^{\uf}(G;\R)\cong\ell^{\infty}(G)$ and $\rho_0$ is just the canonical inversion
\begin{align*}
\ell^{\infty}(G) &\longrightarrow \ell^{\infty}(G)\\
\varphi &\longmapsto (g\longmapsto \varphi(g^{-1})).
\end{align*}
Therefore the transfer map $\overline{m}_{0}$ on $C_0^{\uf}(G;\R)$ coincides with the corresponding right-invariant mean.
\end{rem}

\begin{def.}
The \emph{uniformly finite homological dimension} of a group $G$ is defined by
\[
\hd_{\uf} G = \sup\{n\in\N \; | \; H_n^{\uf}(G;\R)\neq 0 \}\in\N\cup\{\infty\}.
\]
\end{def.}

Because uniformly finite homology is invariant under quasi-isometry, we obtain the following fact as an immediate corollary of Proposition \ref{thm: mean}:

\begin{cor}
 The Hirsch rank of a finitely generated virtually nilpotent group $G$ equals $\hd_{\uf}(G)$ and is therefore a quasi-isometry invariant.
\end{cor}

\begin{proof}
We can assume $G$ to be nilpotent; indeed, by definition, the Hirsch rank of a virtually nilpotent group is the Hirsch rank of any nilpotent subgroup $G'\leq G$ of finite index, and since $G$ is finitely generated it is quasi-isometric to $G'$. Clearly $\hd_{\uf} G\leq \hd_\R G$, where $\hd_\R$ denotes the homological dimension of $G$.

 Conversely, for a finitely generated nilpotent group the homological dimension coincides with the largest integer $n$ for which $H_n(G;\R)\neq 0$ and with its Hirsch rank~\cite{ST}. In view of the inclusion  $H_\ast(G;\R)\longhookrightarrow H_\ast^{\uf}(G;\R)$ obtained in Proposition \ref{thm: mean}, this integer must be smaller or equal than $\hd_{\uf} G$.
\end{proof}

\subsection{Degree zero}
\begin{def.}
 We call the subspace
\begin{align*}
 \widehat{H}_0^{\uf}(G;\R):= \{c\in H_0^{\uf}(G;\R) \mid \fa{m\in M(G)} \overline{m}_0(c) = 0\}
\end{align*}
the \emph{mean-invisible part of} $H_0^{\uf}(G;\R)$.
\end{def.}

\begin{thm}\label{zerocase}
 Let $G$ be a finitely generated infinite amenable group. Then \[ \dim_\R H_0^{\uf}(G;\R)/\widehat{H}_0^{\uf}(G;\R) = \infty.\]

In particular, $\dim_\R H_0^{\uf}(G;\R) = \infty$.
\end{thm}
\begin{proof}
By definition we have the canonical inclusion
\begin{align*}
 LM(G) &\subseteq \ell^{\infty}(G)^\ast =C_0(G;\ell^{\infty}(G))^\ast.
\intertext{Since this inclusion corresponds to the construction of the chain maps in Proposition~\ref{thm: mean} in degree 0, this induces a well-defined injection}
 LM(G)&\longhookrightarrow H_0(G;\ell^{\infty}(G))^\ast\\
 m&\longmapsto ([g\otimes\varphi] \longmapsto m(\varphi)).
\intertext{In view of Remark~\ref{uf-infty} and Remark~\ref{t:Mattgias}, this translates into an inclusion}
 LM(G)&\longhookrightarrow H_0^{\uf}(G;\R)^\ast\\
 m&\longmapsto ([c] \longmapsto \overline{m}_0(c)).
\intertext{By the definition of $\widehat{H}_0^{\uf}(G;\R)$ this also induces an injection}
 LM(G)&\longhookrightarrow \Bigl(H_0^{\uf}(G;\R)/\widehat{H}_0^{\uf}(G;\R)\Bigr)^\ast.
\end{align*}
Hence, by Theorem~\ref{t:Chou} we have $\dim_\R H_0^{\uf}(G;\R)/\widehat{H}_0^{\uf}(G;\R) = \infty.$\qedhere
\end{proof}

\subsection{Higher degrees}
Now we consider uniformly finite homology in higher degrees. We construct infinitely many different classes coming from elements in $\ell^{\infty}(G)$ that are invariant with respect to the action of an infinite index subgroup.

\begin{thm}
\label{not-normal}
Let $n\in\N$ and let $G$ be a finitely generated amenable group. Let $H\leq G$ be a subgroup such that $[G:H]=\infty$ and that the map
\begin{align*}
 H_n(i)\colon H_n(H;\R)\longrightarrow H_n(G;\R)
\end{align*}
induced by the inclusion $i:H\longhookrightarrow G$ is non-trivial. Then $\dim_\R(H_n^{\uf}(G;\R)) = \infty$. 
\end{thm}

The idea behind the proof of Theorem \ref{not-normal} is to construct a family of infinitely many means on $G$ that can be distinguished by a family of $H$-invariant functions (Theorem~\ref{t:infinitemeans}). Mitchell~\cite[Theorem 7]{M} gives a condition for a subset $S\subset G$ to support a left $G$-invariant mean~$m$ such that $m(\chi_S)=1$. We construct infinitely many such subsets and we separate them using Lemma~\ref{separate}. We will also give an alternative proof of Theorem~\ref{not-normal} independent of the result of Mitchell in the case that $H$ is a normal subgroup of $G$.

\begin{lemma}\label{separate}
Let $G$ be an amenable group and $H\leq G$ such that $[G:H]=~\infty$. Let $\pi\colon G\longrightarrow H\backslash G$ be the canonical projection. Then for any pair of finite subsets $T, T' \subseteq G$ there exists $g\in G$ such that $\pi\bigl(T\cdot g\bigr)\cap\pi\bigl(T'\bigr)=\emptyset$.
\end{lemma}

\begin{proof}
For each $g\in G$ such that  $\pi(T\cdot g)\cap\pi(T')\neq\emptyset$ there exist $t\in T$ and $t' \in T'$ such that $tg\in Ht'$, hence $g\in t^{-1}Ht'$. Thus, if $\pi(T\cdot g)\cap\pi(T')\neq\emptyset$ for all $g\in G$, then $G=T^{-1}HT'$. Since $G$ is amenable, there exists a bi-$G$-invariant mean $m$, inducing a finitely additive probability measure $\mu$ on the power set of $G$, given by $\mu(A)=m(\chi_A)$ \cite[Proposition 4.4.4]{CC}. Moreover,  since $H$ has infinite index in $G$ it follows that $\mu(H)=0$. Then we have:

\[
1=\mu(G)=\mu(T^{-1}HT')\leq \sum_{t\in T, t'\in T'}\mu(t^{-1}Ht')=\sum_{t\in T, t'\in T'}\mu(H)=0,
\]
hence a contradiction.\qedhere

\end{proof}

\begin{thm}
\label{t:infinitemeans}
Let $G$ be a finitely generated amenable group and $H\leq G$ a subgroup such that $[G:H]=\infty$. Then there exists an infinite family $\left (m_j\right)_{j\in J}$ of left $G$-invariant means and an infinite family $\left (f_j\right)_{j\in J}$  of (left) $H$-invariant functions in $\ell^{\infty}(G)$, such that $m_k(f_j)=\delta_{k,j}$ for any $k,j\in J$.
\end{thm}

\begin{proof}
Let $n\in \N$ and let $\pi\colon G\longrightarrow H\backslash G$ be the canonical projection. Consider $G$ equipped with the word metric with respect to a (finite) set of generators. We inductively construct finite sets $A^{k}_{l}\subseteq G$ for all $k\in \{1, \dots, n\}$ and for all $l\in \N$ such that

\begin{itemize}

\item The family $\left ( \pi(A^{k}_{l})\right)_{k\in \{1, \dots, n\}, l\in \N}$ is pairwise disjoint.

\item For all  $k\in \{1, \dots, n\}$ and $l\in \N$ there exists $g\in G$ such that $A^{k}_{l}=B_{l}(e)\cdot g$, where $B_{l}(e)$ denotes the ball of radius $l$ centered at the identity element.

\end{itemize}

Let $A^{1}_{1}:=B_{1}(e)$; assume that the sets have been constructed for all indices smaller or equal than $(l,k)$ (using lexicographic order). Then:

\begin{itemize}

\item If $k< n$,  by Lemma~\ref{separate} there exists $g\in G$ such that 

\[
\pi\Biggl(\bigcup_{(l',k')\leq (l,k)}A^{k'}_{l'}\Biggr)\cap\pi\bigl(B_{l}(e)\cdot g\bigr)=\emptyset.
\]

Set $A^{k+1}_{l}:=B_{l}(e)\cdot g$.

\item If $k=n$, by Lemma~\ref{separate} there exists $g\in G$ such that 

\[
\pi\Biggl(\bigcup_{(l',k')\leq(l,k)}A^{k'}_{l'}\Biggr)\cap\pi\bigl(B_{l+1}(e)\cdot g\bigr)=\emptyset.
\]

Set $A^{1}_{l+1}:=B_{l+1}(e)\cdot g$.

\end{itemize}
 
Now set for all $k\in \{1,\dots, n\}$

\[
T^{k}:=\bigcup_{l\in \N}HA^{k}_{l}.
\]
Notice that $T^1, \dots, T^n$ are pairwise disjoint by construction.

For any finite subset $F\subseteq G$, there exists $r\in \N$ such that $F\subseteq B_{r}(e)$, hence, for all~$k\in\{1, \dots,n\}$ there exists $g\in G$ such that $F\cdot g\subseteq B_{r}(e)\cdot g\subseteq T^{k}$. This implies that each $T^{k}$ is left-thick in the sense of Mitchell~\cite{M}. Hence, ~\cite[Theorem 7]{M}, for each~$k\in\{1, \dots,n\}$ there exists a left $G$-invariant mean $m_k$, such that $m_k(\chi_{T^k})=1$. Moreover, for each $k\in\{1, \dots,n\}$ and for each $j\in\{1, \dots,n\}\backslash\{k\}$ we have that $m_k(\chi_{T^j})=0$ since the sets $T^1, \dots, T^n$ are pairwise disjoint and $m_k(G)=1$. 
By definition of $T^k$, the functions $\chi_{T^k}$ are $H$-invariant for any $k\in\{1, \dots,n\}$.

We have, then, constructed a family of left $G$-invariant means $\{m_k\}_{k\in\{1, \dots,n\}}$ and a family of bounded $H$-invariant functions $\{\chi_{T^k}\}_{k\in\{1, \dots,n\}}$ on $G$. We can repeat the same construction for any $n\in \N$, so we can have an arbitrary large finite family of means and of functions satisfying the theorem. Using a slightly different induction step, it is possible to construct a family of finite sets $A^{k}_{l}\subseteq G$ for all $k\in\N$ and for all $l\in\N$ as above. In this way we are able to construct directly an infinite family of means and of functions satisfying the theorem.
\end{proof}

We are now ready to prove our main theorem. 

\begin{proof}[First proof of Theorem~\ref{not-normal}]
If $\dim_{\R}H_n(G;\R)=\infty$, then by Proposition~\ref{thm: mean}, we immediately conclude that $H_{n}^{\uf}(G;\R)$ must be infinite dimensional.

Therefore, we assume $\dim_{\R}H_n(G;\R)<\infty$. We will prove the statement for  $H_n(G;\ell^{\infty}(G))$ and in view of Remark ~\ref{uf-infty}, the theorem will also follow for $H_n^{\uf}(G;\R)$. 

Let $c\in C_n(H;\R)$ be a cycle such that $[i(c)]\in H_n(G;\R)$ is non-trivial. Consider the space $\ell^{\infty}(G)^{H}:=\bigl\{\varphi\in\ell^{\infty}(G) \; \bigm | \; \fa{h\in H}  h\cdot\varphi=\varphi \bigr\}$. Then
\begin{align*}
S_c :=\bigl\{ i(c)\otimes \varphi\in C_{n}(G;\ell^{\infty}(G))\; \bigm | \; \varphi\in\ell^{\infty}(G)^H \bigr\} 
\end{align*}
is a subspace of cycles in $C_{n}(G;\ell^{\infty}(G))$. Indeed, for every $\varphi\in\ell^{\infty}(G)^H$ one can define the $H$-equivariant map $f_{\varphi}\colon\R\longrightarrow\ell^{\infty}(G), \ 1\longmapsto\varphi$, which induces a chain map $C_{n}(i,f_{\varphi}): C_{n}(H;\R)\longrightarrow C_{n}(G;\ell^{\infty}(G))$ that maps $c$ to $i(c)\otimes \varphi$. 

Let  $\left (m_j\right)_{j\in J}$ and $\left (f_j\right)_{j\in J}$ be as in Theorem~\ref{t:infinitemeans}; since for all $j\in J$, the map $f_j$ is $H$-left invariant, the elements $\bigl(i(c)\otimes f_{j}\bigr)_{j\in J}\in C_n(G;\ell^{\infty}(G))$ belong to $S_c$, so they are all cycles in $C_n(G;\ell^{\infty}(G))$. 
Now consider the family of induced transfer maps $\bigl(m_{{j}{\ast}}\bigr)_{j\in J}\in\Hom_{\R}\bigl(H_n(G;\ell^{\infty}(G)),H_n(G;\R)\bigr)$. For any $k,j\in J$, by Theorem~\ref{t:infinitemeans}, $m_{{k}{\ast}}([i(c)\otimes f_j])=\delta_{i,j}\cdot [i(c)]$, hence the family $\bigl(m_{{j}{\ast}}\bigr)_{j\in J}$ is linearly independent in $\Hom_{\R}\bigl(H_n(G;\ell^{\infty}(G)),H_n(G;\R)\bigr)$. Therefore $H_n(G;\ell^{\infty}(G))$ is infinite dimensional.
\end{proof}

Now we give a more direct proof of the main result under the additional assumption that $H$ is normal in $G$. 

\begin{proof}[Second proof of Theorem~\ref{not-normal}]

Let $H$ now be normal in $G$. Since $H$ is a subgroup of the amenable group $G$, it is also amenable. Let $m_0$ be a left $H$-invariant mean. Consider the following transfer map 
\begin{align*}
\tau\colon \ell^{\infty}(G)&\longrightarrow \ell^{\infty}(G/  H)\\
\varphi &\longmapsto \left(g H \longmapsto m_0((g^{-1}\cdot\varphi)|_{H })\right).
\end{align*}
For any $\varphi\in\ell^{\infty}(G)$, the map $\tau(\varphi)$ is well-defined. Indeed, let $g_1, g_2$ be elements of $G$ such that $g_{1}H=g_{2}H$; then $g_2=g_{1}h$ for some $h\in H$, hence
\[
\tau(\varphi)(g_{2}H)=m_{0}((g_{2}^{-1}\cdot\varphi)|_{H})=m_{0}(({h}^{-1}g_{1}^{-1}\cdot\varphi)|_{H})=m_{0}((g_{1}^{-1}\cdot\varphi)|_{H})=\tau(\varphi)(g_{1}H)
\]
by the $H$-invariance of $m_0$. It is clear that $\tau(\varphi)$ is bounded for any $\varphi\in\ell^{\infty}(G)$.

We also have a map induced by the canonical projection $\pi\colon G\longrightarrow G / H$
\begin{align*}
\pi^{\ast}\colon \ell^{\infty}(G/H)&\longrightarrow \ell^{\infty}(G)\\
\psi &\longmapsto \psi\circ\pi.
\end{align*}
It is easy to see that $\Ima(\pi^\ast)\subseteq\ell^{\infty}(G)^{H}$; indeed, for any $\psi\in \ell^{\infty}(G/H)$, any $h\in H$ and any $g\in G$,
\[
h\cdot\pi^\ast(\psi)(g)=\pi^\ast(\psi)(h^{-1}g)=\psi(h^{-1}gH)=\psi(gH).
\]
Notice that the equality $h^{-1}gH=gH$ holds since we have assumed $H$ to be normal in $G$. The following composition
\begin{align*}
\ell^{\infty}(G/H)&\xrightarrow{\pi^{\ast}} \ell^{\infty}(G)^{H}\xrightarrow{\tau{|}_{\ell^{\infty}(G)^{H}}}\ell^{\infty}(G/H)
\end{align*}
is the identity on $\ell^{\infty}(G/H)$: Indeed, for any $\psi\in \ell^{\infty}(G/H)$ and any class $gH\in G/H$ we have
\[
\tau(\psi\circ\pi)(gH)=m_{0}(g^{-1}\cdot(\psi\circ\pi))|_{H})=m_{0}(h\mapsto\psi(gH))=\psi(gH),
\]
since the function $h\mapsto\psi(gH)$ is constant. In particular, $\tau{|}_{\ell^{\infty}(G)^{H}}$ is surjective.
It is easy to see that $\tau(\chi_G)=1$ and that $\tau(\varphi)\geq 0$ for any $\varphi\geq 0$ in $\ell^{\infty}(G)$. Moreover, $\tau$ is $G$-equivariant: indeed, for any $g, g'\in G$ and any $\varphi\in\ell^{\infty}(G)$, we have:
\[
\begin{split}
\tau(g'\cdot\varphi)(gH)& = m_0((g^{-1}\cdot g'\cdot\varphi)|_H)\\
 & =m_0((({g'}^{-1}g)^{-1}\cdot\varphi)|_H)\\
 & =\tau(\varphi)({g'}^{-1}gH)\\
 & ={g'}\cdot\tau(\varphi)(gH).
\end{split}
\]
By dualising $\tau$ we get a map
\begin{align*}
\tau^{\ast} \colon \ell^{\infty}(G / H)^{\ast} &\longrightarrow \ell^{\infty}(G)^{\ast}\\
F &\longmapsto F \circ \tau.
\intertext{Since $H$ is a normal  subgroup of $G$ with infinite index, $G/H$ is also an infinite amenable group, so restricting $\tau^{\ast}$ to $LM(G/H)$ we have a well-defined map}
\tau^{\ast} \colon LM(G / H) &\longrightarrow LM(G):
\end{align*}
Indeed, let $m\in M(G/H)$ be a mean on $G/H$. Then, since $\tau(\chi_G)=1$ and $\tau(\varphi)\geq 0$ for any $\varphi\geq 0$ in $\ell^{\infty}(G)$, it immediately follows that $\tau^{\ast}(m)$ satisfies condition (i) and (ii) of Definition~\ref{definitionmean}. Condition (iii) follows by the $G$-equivariance of $\tau$, which implies that $\tau^{\ast}(m)$ is a left $G$-invariant mean.

Let $[S_c]$ be the space of classes represented by elements of $S_c$ as defined in the first proof of Theorem~\ref{not-normal}.
Now consider the following diagram.
\begin{center}
\begin{tikzpicture}
\matrix (m) [scale = 0.25,matrix of math nodes, row sep=3.5em,
column sep=2em, text height=1.5ex, text depth=1.5ex]
{LM(G)&\Hom_{\R}\bigl(H_n(G;\ell^{\infty}(G)),H_n(G;\R)\bigr)\\
LM(G / H)&\Hom_{\R}\bigl([S_c],H_n(G;\R)\bigr).
\\ };
\path[->]
(m-1-1) edge node[auto] {$ \Phi $}  (m-1-2)
(m-2-1) edge node[auto] {$ \tau^\ast $}  (m-1-1)
(m-1-2) edge node[auto] {$\res $}  (m-2-2)
(m-2-1) edge node[below] {$ \overline{\Phi} $}  (m-2-2);
\end{tikzpicture}
\end{center}
where the upper horizontal map $\Phi$ takes every mean $m\in M(G)$ to the induced homomorphism $m_\ast$. Here $\res$ denotes the restriction map and the lower horizontal map is given by composition. 

Since $H_n(G;\R)$ can be assumed to be finite dimensional, to prove the theorem it suffices to show that $\overline{\Phi}$ is injective; indeed, by Theorem ~\ref{t:Chou} of Chou, since $G/H$ is infinite amenable, $LM(G/H)$ is infinite dimensional and from the injectivity of $\overline{\Phi}$ it would follow that $[S_c]\subseteq H_n(G;\ell^{\infty}(G))$ must be infinite dimensional.

We now show that $\overline \Phi$ is injective. So, consider two means $m_1\neq m_2 \in LM(G/H)$;  then there exists $\psi\in\ell^{\infty}(G/H)$ such that $m_1(\psi)\neq m_2(\psi)$. By surjectivity of $\tau{|}_{\ell^{\infty}(G)^H}$ we know that there exists $\varphi\in\ell^{\infty}(G)^H$ such that $\tau(\varphi)=\psi$. Consider now the image of $[i(c)\otimes\varphi]\in [S_c]$ by the homomorphisms $\overline{\Phi}(m_1), \overline{\Phi}(m_2)\in\Hom([S_c],H_n(G;\R))$. We have for $k\in\{1,2\}$
\[
\overline{\Phi}(m_k)([i(c)\otimes\varphi])=\Phi(\tau^\ast(m_k))([i(c)\otimes\varphi])=[m_k(\tau(\varphi))\cdot i(c)]=m_k(\psi)\cdot[ i(c)].
\]
Since $m_1(\psi)\neq m_2 (\psi)$ and $[i(c)]\neq 0$, the two classes $m_1(\psi)\cdot [i(c)]$ and $m_2(\psi)\cdot [ i(c)]$ are also different in $H_n (G;\R)$, which implies that $\overline{\Phi}(m_1)\neq \overline{\Phi}(m_2)$; in particular, $\overline{\Phi}$ is injective.
\end{proof}

\section{Examples}\label{examples}

\begin{cor}\label{c:onecase}
 Let $G$ be a finitely generated amenable group. Assume that $H_1(G;\R)$ is non-trivial, i.e., that the abelianization of $G$ is not a torsion group. Then 
\begin{align*}
 H_1^{\uf}(G;\R)\cong \begin{cases}
                \R &\text{ if $G$ is virtually $\Z$}\\
		\text{infinite dimensional} &\text{ otherwise.}
               \end{cases}
\end{align*}
\end{cor}
\begin{proof}
Let $g\in G$ be an element such that $(1,g)\in H_1(G;\R)$ is a non-trivial cycle. Such a $g$ exists by assumption and the isomorphism $H_1(G;\R)\cong G_\text{ab}\otimes \R$. 
If $G$ is not virtually $\Z$, the claim follows from Theorem~\ref{not-normal}, since then $[G:\langle g\rangle]=\infty$. And for the virtually $\Z$ case see the following example. \qedhere
\end{proof}

\begin{exa}
 Let $G$ be a finitely generated infinite amenable group. Consider  semi-direct products of the form $G\rtimes \Z^l$ for $l\in\N$. Then for all $k\in\{0,\dots,l\}$
\[
 \dim_\R H_k^{\uf}(G\rtimes \Z^l;\R) =\infty.
\]
Similar results hold if one replaces $\Z^l$ by an amenable group with non-vanishing homology in the correct degrees.  
In particular, for all $l\in\N$ 
\begin{align*}
 H^{\uf}_k(\Z^l;\R) =
 \begin{cases}
  \R &\text{ if $k=l$}\\
  \text{infinite dimensional} &\text{ if  $k\in\{0,\dots,l-1\}$}\\
 0 &\text{ else.}
 \end{cases}
\end{align*}
\end{exa}
\begin{proof}
 The splitting map $\Z^l\longhookrightarrow G\rtimes \Z^l$ induces a non-trivial map in degree $0,\dots,l$ in homology and hence the first part follows from Theorem~\ref{not-normal}. The product $\Z^l=\Z^{l-1}\times\Z$ is a special case. The degrees $k\geq l$ follow because the group is a Poincar\'e duality group and $H^0(G;\ell^{\infty}(G)) \cong \ell^{\infty}(G)^G \cong \R$ for all groups $G$. 
\end{proof}
\begin{exa}
 For the integral three-dimensional Heisenberg group $\Heis_3$ we get:
\begin{align*} 
 H_k^{\uf}(\Heis_3;\R) = 
 \begin{cases}
  \R &\text{ if $k=3$}\\
  \text{infinite dimensional} &\text{ if  $k \in \{0,1,2\}$}\\
 0 &\text{ else.}
 \end{cases}
\end{align*}
\end{exa}
\begin{proof}
 We only have to consider $k=2$, the cases $k=0,1$ follow directly from Theorem~\ref{zerocase} and Corollary~\ref{c:onecase}. The higher degrees are a consequence of Poincar\'e duality. Consider the presentation $\Heis_3 \cong \langle x,y,z \mid [x,y],[y,z], y^{-1}[x,y]\rangle.$ By Hopf's Theorem,~\cite[II.5, Theorem 3]{B},to compute the second homology it suffices to look at the Schur multiplier which, in this case, is generated by the symbols $[x,y],[y,z]$. Hence the inclusion of the subgroup generated by $x$ and $y$ is non-trivial in homology in degree 2, so we can apply Theorem~\ref{not-normal} again.
\end{proof}
Actually, the last examples are simple special cases of a more general result:
\begin{exa}\label{e:nilpotent}
Let $G$ be a finitely generated virtually nilpotent group of Hirsch rank $h\in\N$. Then
\begin{align*} 
 H_k^{\uf}(G;\R) = 
 \begin{cases}
  \R &\text{ if $k=h$}\\
  \text{infinite dimensional} &\text{ if  $k \in \{0,\dots, h-1\}$}\\
 0 &\text{ else.}
 \end{cases}
\end{align*}
\end{exa}
\begin{proof}
 We closely follow the calculation of the homology groups of nilpotent groups of Baumslag, Miller and Short~\cite{BMS} to see that our condition in Theorem~\ref{not-normal} is satisfied in degree $0,\dots, h-1$.  After passing to a finite index subgroup, we may assume that $G$ is torsion-free and nilpotent. 
 Then we can write $G$ as a split extension of the form
\[
 1\longrightarrow N \longrightarrow G\longrightarrow \Z\longrightarrow 1
\]
for $N\subset G$ a normal subgroup of Hirsch rank $h-1$. The Hochschild-Serre spectral sequence for this split extension induces a short exact sequence 
\[
 0\longrightarrow H_0(\Z;H_i(N;\R))\longrightarrow H_i(G;\R) \longrightarrow H_1(\Z;H_{i-1}(N;\R))\longrightarrow 0
\]
for all $i\in\N_{>0}$.
 The map on the left hand side is one edge map of the spectral sequence and is induced by the canonical map $H_i(N;\R) \longrightarrow H_i(G;\R)$ under the identification $ H_i(N;\R)_\Z \cong H_0(\Z,H_i(N;\R))$~\cite{Wei}. In particular, the canonical map $H_i(N;\R) \longrightarrow H_i(G;\R)$ is non-trivial if $ H_i(N;\R)_\Z $ is non-trivial. But these homology groups are non-trivial~\cite[Proof of Theorem 16]{BMS}. The degrees $k\geq h$ follow because finitely generated nilpotent groups are Poincar\'e duality groups~\cite[VIII.10, Example 1]{B}.
\end{proof}
\begin{exa}
Consider $A\in \Sl(2,\Z)$. Then for the semi-direct product $\Z^2\rtimes_{A}\Z$ given by the action of $\Z$ on $\Z^2$ induced by A, we have
\begin{align*} 
 H_k^{\uf}(\Z^2\rtimes_{A} \Z;\R) = 
 \begin{cases}
  \R &\text{ if $k=3$}\\
  \text{infinite dimensional} &\text{ if  $k \in \{0,1,2\}$}\\
 0 &\text{ else.}
 \end{cases}
\end{align*}
In particular, this example includes cocompact lattices in Sol~\cite{S}. 
\end{exa}
\begin{proof}
 This is similar to~\ref{e:nilpotent}. We only have to consider $k=2$ the other cases follow as in the other examples. By the Hochschild-Serre spectral sequence we get a short exact sequence 
\[
  0\longrightarrow H_0(\Z,H_2(\Z^2;\R))\longrightarrow H_2(\Z^2\rtimes_{A}\Z;\R) \longrightarrow H_1(\Z,H_{1}(\Z^2;\R))\longrightarrow 0.
\]
By Poincar\'e duality and the universal coefficient theorem, the dimension of the middle term is equal to $\dim_\R H_1(\Z\rtimes_{A}\Z;\R)=3$. Since the dimension of the term at the right-hand-side is at most 2, the map $H_0(\Z,H_2(\Z^2;\R))\longrightarrow H_2(\Z^2\rtimes_{A}\Z;\R)$ is non-trivial. In particular the canonical map $H_2(\Z^2;\R)\longrightarrow H_2(\Z^2\rtimes_{A}\Z;\R)$ is also non-trivial. Hence $H_2^{\uf}(\Z^2\rtimes_{A} \Z;\R)$ is infinite dimensional  by Theorem~\ref{not-normal}. 
\end{proof}

There are many well-known examples of finitely generated amenable groups having non-trivial real homology in each degree. If $G$ is such a group, the group $G\times \Z$ satisfies $\dim_{\R} H_n^{\uf}(G\times \Z;\R) =\infty$ for all $n\in\N$ by Theorem~\ref{not-normal}. For instance:

\begin{exa}
There exists a finitely-presented metabelian group $G$ such that for all $n\in\N$
\[
  \dim_\R H^{\uf}_n(G;\R) =\infty.
\]
\end{exa}
\begin{proof}
 For example, Baumslag and Dyer have shown~\cite{BD}, that Baumslag's met\-abelian group $B:=\langle a,s,t\mid a^s = a\cdot a^t, [a,a^t] =1 = [s,t]\rangle$ has non-trivial homology in degree $n\geq 3$. Set $G:= B\times \Z^3$. 
\end{proof}

\begin{exa}
Consider Thompson's group \text{$F:= \langle x_0,x_1,\dots\mid x_n^{x_i} = x_{n+1} \text{ for } i<n\rangle$.} We have 
\[
 F \text{ amenable} \Longrightarrow \fa{n\in\N} \dim_{\R} H_n^{\uf}(F;\R) = \infty.
\]
\end{exa}
\begin{proof}
 We follow Kenneth Brown's calculation of the homology of $F$ ~\cite{B2}. Consider $F$ as the group of dyadic piecewise linear homeomorphisms $[0,1]\longrightarrow [0,1]$. Brown notes that in this description the commutator subgroup of $F$ can be seen as $F'=\{f\in F\mid f'(0) = 1= f'(1)\}$. There is a product map
\begin{align*}
\ast \colon F\times F&\longrightarrow F\\
 (f,g) &\longmapsto \left(t\longmapsto \begin{cases} f(2\cdot t)/2 & 0\leq t\leq 1/2\\ g(2\cdot t-1)/2 & 1/2\leq t\leq 1.\end{cases}\right)                                                                                                                                                                                                                                                                                                                                                                                                                                                                                                                                                                                                                                   \end{align*}
Define the subgroup $G:= \{f\in F\mid f'(0) =1\}\leq F$. The product map restricts both to $F'$ and $G$ and induces a natural product in homology on $H_\ast(F';\R),H_\ast(G;\R)$ and $H_\ast(F;\R)$. 

Brown~\cite[Theorem 4.1]{B2} shows that $H_\ast(F;\R)$ is generated as a ring by elements $\varepsilon,\alpha,\beta$; that alternating products $\alpha\cdot\beta\cdot \alpha\cdots$ and $\beta\cdot\alpha\cdot\beta\cdots$ form a basis of $H_\ast(F;\R)$ in positive degree and that the image of $H_\ast(F';\R)\longrightarrow H_\ast(F;\R)$ contains $\alpha\cdot \beta\in H_2(F;\R)$ and hence is non-trivial in even degrees.

 Since $F'$ is strictly contained in $G$, the image of the map $H_\ast(G;\R)\longrightarrow H_\ast(F;\R)$ contains a non-trivial linear combination of $\alpha$ and $\beta$ (since they form a basis of~$H_1(F;\R)$), hence the map $H_\ast(G;\R)\longrightarrow H_\ast(F;\R)$ is non-trival in each degree and we may apply Theorem~\ref{not-normal}.
\end{proof}

\section{Sparse classes in $H_0^{\uf}(G;\R)$}\label{sparse}
In this section, we give a geometric condition for classes in $H^{\uf}_0(G;\R)$ to be mean-invisible and show that in an infinite amenable group, there are infinitely many linear independent classes of this type. In particular, we show:

\begin{thm}\label{t:sparse}
 Let $G$ be a finitely generated infinite amenable group. Then \[\dim_{\R} \widehat{H}_0^{\uf}(G;\R) =\infty.\] 
\end{thm}

In this section, we always consider finitely generated groups equipped with a word metric with respect to some finite generating set (Definition~\ref{d:wordmetric}). 
\subsection{Distinguishing classes by asymptotic behavior}
Let $G$ be a finitely generated discrete group with a word metric~$d$. For any $S\subseteq G$ let
\[
\partial_r(S):=\{g\in G \; \bigm | \; 0<d(g,S)\leq r \}.
\]

The following theorem due to Whyte \cite[Theorem 7.6]{W} gives a characterization of trivial classes in uniformly finite homology in degree 0.

\begin{thm}\label{t:whyte}
Let $G$ be a finitely generated group. Let $c$ be a cycle in $C_0^{\uf}(G;\R)$. Then $[c] = 0\in H_0^{\uf}(G;\R)$ if and only if
\[
 \exi{C,r\in \N_{>0}}\fa{S\subseteq G \text{ finite}}  \left|\sum_{s\in S}c(s)\right| \leq C\cdot |\partial_r S|. 
\]
\end{thm}
\begin{rem}
Whyte states Theorem~\ref{t:whyte} for uniformly finite homology with $\Z$ coefficients. However, for infinite finitely generated groups in degree zero, there is an isomorphisms of uniformly finite homology with $\Z$ coefficients and uniformly finite homology with $\R$ coefficients induced by the canonical inclusion $\Z\longhookrightarrow\R$ \cite[Lemma 7.7]{W}. It is easy to see that this isomorphism translates the original statement of Whyte to Theorem~\ref{t:whyte}. 
\end{rem}

Whyte's criterion can be reformulated in terms of the asymptotic behavior of a degree zero chain in comparison to the behavior of a F\o lner sequence. First, we introduce a notion to compare the asymptotic behavior of functions: 
\begin{def.}Let $\alpha,\beta \colon \N\longrightarrow \R_{>0}$ be two functions.
\begin{enumerate}
\item We write $\alpha\prec \beta$ if \[\lim_{n\to \infty} \frac{\alpha(n)}{\beta(n)}=0.\]
\item We write $\alpha\sim \beta$ if 
\[
 \lim_{n\to \infty} \frac{\alpha(n)}{\beta(n)} \in \R_{>0}.
\]
\item We also write $\alpha\preceq\beta$ if $\alpha\sim\beta$ or $\alpha \prec \beta$. 
\end{enumerate}
\end{def.}

We recall the definition of a F\o lner sequence for finitely generated groups. Also recall that a finitely generated group is amenable if and only if it admits a F\o lner sequence~\cite[Theorem 4.9.2]{CC}.

\begin{def.}
Let $G$ be a fininitely generated group. A \emph{F\o lner sequence} in~$G$ is a sequence~$(S_j)_{j\in\N}$ of non-empty finite subsets of $G$ such that for each $r\in\R_{>0}$
\[
\lim_{j\to\infty}\frac{|\partial_{r}(S_j)|}{|S_j|}=0.
\]
\end{def.}

Now fix a F\o lner sequence $S:= (S_j)_{j\in\N}$ of a finitely generated amenable group~$G$. 

For all $c\in\ell^{\infty}(G)$ we define a function 
\begin{align*}
 \beta_c^{S} \colon \N &\longrightarrow \R\\
n&\longmapsto \frac{|\sum_{s\in S_n}c(s)|}{|S_n|}.
\intertext{
Finally, we also consider the behavior of the boundaries}
 \sigma_S\colon \N&\longrightarrow \R_{>0}\\
 n&\longmapsto \frac{|\partial_1 S_n|}{|S_n|}.
\end{align*}
By comparing the behavior of these functions, we can distinguish between classes in $H_0^{\uf}(G;\R)$: 
\begin{lemma}\label{l:growth} Let $G$ be a finitely generated amenable group and $S$ a F\o lner sequence in $G$. Let $c\in\ell^{\infty}(G)$. 
 \begin{enumerate}
  \item We have $0\preceq \beta^S_c \preceq 1$. 
\item If $\beta^S_c \succ \sigma_S$, then $[c]\neq 0\in H_0^{\uf}(G;\R)$.
\item Conversely, if $[c]\neq 0\in H_0^{\uf}(G;\R)$ then there exists a F\o lner sequence $S'$ such that $\sigma_{S'}\prec \beta_c^{S'}$. 
\item More generally: If $(c_n)_{n\in\N}$ is a sequence in $\ell^{\infty}(G)$ such that $\sigma_S\prec\beta^S_{c_0}$ and 
\[
 \fa{n\in\N} \beta^S_{c_n}\prec \beta^S_{c_{n+1}},
\]
 then the family $([c_n])_{n\in\N}$ of classes in $H_0^{\uf}(G;\R)$ is linearly independent. 
 \end{enumerate}
\end{lemma}
\begin{proof}\hfill
 \begin{enumerate}
  \item This is obvious by definition.
 \item Assume $[c]=0$. By Whyte's criterion there exist $C,r\in\N_{>0}$ such that for all $n\in\N$
\begin{align*}
 0<\frac1C\leq \frac{|\partial_r S_n|}{\left|\sum_{s\in S}c(s)\right|}\leq |B_r(e)|\cdot \frac{|\partial_1 S_n|}{\left|\sum_{s\in S}c(s)\right|}.
\end{align*}
Hence $\beta^S_c\not\succ \sigma_S$. 
\item By Whyte's criterion, if $[c]\neq 0$,  for each $n\in\N$ there exists a finite subset $S'_n\subseteq G$  such that
\begin{align*}
 \Bigl|\sum_{s\in S'_n}c(s)\Bigr|> n\cdot |\partial_1 S'_n|;
\intertext{hence}
\frac{1}{\|c\|_{\infty}}\cdot \frac{|\partial_1 S'_n|}{ |S'_n|}\leq \frac{|\partial_1 S'_n|}{ \left|\sum_{s\in S'_n}c(s)\right|}<\frac1n.
\end{align*}
In particular $S':=(S'_n)_{n\in\N}$ is a F\o lner sequence and $\sigma_{S'}\prec \beta_c^{S'}$.
\item For each $c\in \ell^{\infty}(G)$ define the subspace
\[
C_0^{\uf}(G;\R)^{\preceq c}:=\Bigl\{c'\in C_0^{\uf}(G;\R)\Bigm | \lim_{n\to\infty} \frac{\sum_{s\in S_n}c'(s)}{\sum_{s\in S_n}c(s)} \text{ exists}\Bigr\}.
\]
Choose a splitting of $\R$-vector spaces $C_0^{\uf}(G;\R)= C_0^{\uf}(G;\R)^{\preceq c}\oplus V$ and define a linear map
\begin{align*}
 \gamma^S_c \colon C_0^{\uf}(G;\R) = C_0^{\uf}(G;\R)^{\preceq c}\oplus V&\longrightarrow \R\\
 (c',c'') &\longmapsto \lim_{n\to \infty} \frac{\sum_{s\in S_n}c'(s)}{\sum_{s\in S_n}c(s)}.
\end{align*}
If $a\in C_0^{\uf}(G;\R)$ is a boundary, the function $\beta_a^S/\sigma_S$ is bounded by Whyte's criterion. Therefore, if $ \sigma_S\prec \beta^S_c$ then $\beta^S_a\prec \beta^S_c$ and hence $\gamma^S_c(a) =0$. Thus $\gamma^S_c$ induces a map 
\[
\overline{\gamma^S_c}\colon H_0^{\uf}(G;\R)\longrightarrow \R.
\]
In our situation we have $\gamma^S_{c_i}(c_j) = \delta_{ij}$ for all $j\leq i$ in $\N$. Hence, $([c_n])_{n\in\N}$ is linearly independent.  \qedhere
 \end{enumerate}
\end{proof}
\begin{exa}
 The sequence $(\{n^k\mid n\in\N\})_{k\in\ \N_{>0}}$ of subsets of $\Z$ satisfies the conditions of the last part of Lemma~\ref{l:growth}, and hence induces a sequence of linearly independent classes in $H_0^{\uf}(\Z;\R)$.
 \end{exa}

\subsection{Sparse Classes}

We now introduce a geometric condition that will ensure that a subset is mean-invisible:
\begin{def.}
 Let $G$ be a finitely generated group. We call a subset $\Gamma \subseteq G$ \emph{(asymptotically) sparse} if 
\begin{align*}
 \exi{C\in\N} \fa{r\in\N_{>0}}\exi{R\in \N_{>0}} \fa{g\in G\setminus B_R(e)} |\Gamma\cap B_r(g)| \leq C.
\end{align*}
\end{def.}

\begin{exa}
 For each $k\in\N_{>1}$ the subsets $\{n^k\mid n\in \N\}\subseteq \Z$ are  sparse. 
\end{exa}

\begin{lemma}
 If $\Gamma \subseteq G$ is  sparse and $G$ a finitely generated infinite amenable group we have
\[
 [\chi_\Gamma] \in \widehat{H}_0^{\uf}(G;\R).
\]
\end{lemma}
\begin{proof}
 Let $m$ be a left-invariant mean on $G$ and $C\in\N$ as in the definition of sparse. Let  $r\in\N_{>0}$. Here we write $\chi(S):=\chi_S$ for the characteristic function of any subset $S\subseteq G$. Since $G$ is infinite we have for all $R\in \N_{>0}$ that $\overline{m}(\chi(\Gamma)) = \overline{m}(\chi(\Gamma\setminus B_R(e)))$. Hence we may assume that for all $g\in G$
\[
|\{\gamma \in \Gamma \mid \exi{h\in B_r(e)} \gamma\cdot h = g\} |= |\Gamma\cap B_r(g)|\leq C.
\]
Therefore the coefficients of $\sum_{h\in B_r(e)} \chi(\Gamma\cdot h)$ are bounded by $C$.
So we see that
\[
|B_r(e)|\cdot \overline{m}(\chi(\Gamma))=\sum_{h\in B_r(e)} \overline{m}(\chi(\Gamma\cdot h))=\overline{m}\left(\sum_{h\in B_r(e)} \chi(\Gamma\cdot h)\right)\leq C.
\]
So $\overline{m}(\chi(\Gamma)) = 0$. 
\end{proof}
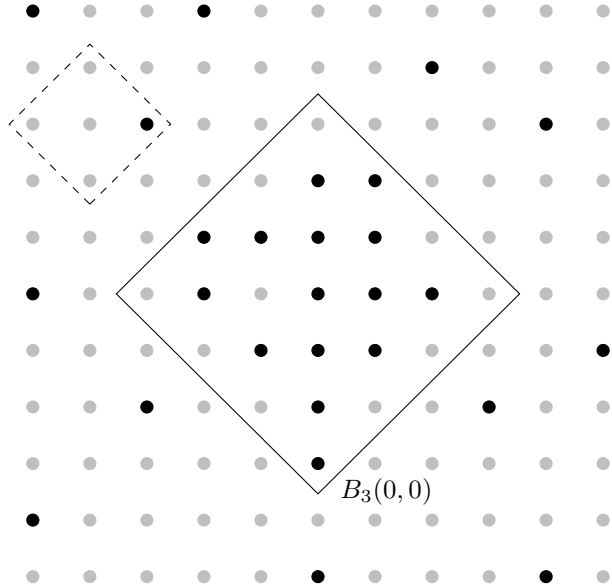
\begin{figure}[here]
\begin{center}
  \begin{tikzpicture}[scale = .75]
\foreach \x in {-5,-4,...,5} {
\foreach \y in {-5,-4,...,5}{
\node (a\x\y) at (\x,\y)  {};
\draw[fill, color= lightgray] (a\x\y) circle (3pt);
}
}
\draw[fill] (a00) circle (3pt);
\draw[fill] (a01) circle (3pt);
\draw[fill] (a10) circle (3pt);
\draw[fill] (a-1-1) circle (3pt);
\draw[fill] (a1-1) circle (3pt);
\draw[fill] (a0-1) circle (3pt);
\draw[fill] (a-20) circle (3pt);
\draw[fill] (a20) circle (3pt);
\draw[fill] (a0-2) circle (3pt);
\draw[fill] (a02) circle (3pt);
\draw[fill] (a11) circle (3pt);
\draw[fill] (a-11) circle (3pt);
\draw[fill] (a-55) circle (3pt);
\draw[fill] (a-25) circle (3pt);
\draw[fill] (a-33) circle (3pt);
\draw[fill] (a0-5) circle (3pt);
\draw[fill] (a-50) circle (3pt);
\draw[fill] (a-3-2) circle (3pt);
\draw[fill] (a-5-4) circle (3pt);
\draw[fill] (a4-5) circle (3pt);
\draw[fill] (a43) circle (3pt);
\draw[fill] (a5-1) circle (3pt);
\draw[fill] (a3-2) circle (3pt);
\draw[fill] (a24) circle (3pt);
\draw[fill] (a-21) circle (3pt);
\draw[fill] (a0-3) circle (3pt);
\draw[fill] (a12) circle (3pt);
\draw[rotate=45] (-2.5,-2.5)rectangle (2.5,2.5);
\draw[dashed, xshift=-4cm,yshift=3cm,rotate=45] (-1,-1)rectangle (1,1);
\draw (1.2,-3.1) node[below]  {$B_3(0,0)$};
 \end{tikzpicture}
\end{center}\caption{The black dots form a sparse subset in $\Z^2$ with the standard word metric.}
\end{figure}
\subsection{Constructing sparse classes}
We recall the notion of tilings, who will be the building blocks for our sparse classes:
\begin{def.}
 Let $G$ be a finitely generated group with the word metric. For $r\in\N_{>0}$ we call a subset $T\subseteq G$ an \emph{$r$-tiling for $G$} if
\begin{enumerate}
\item $ \fa{g_1,g_2\in T} B_r(g_1)\cap B_r(g_2)\neq \emptyset \Longrightarrow g_1=g_2$
\item $G= \bigcup_{g\in T} B_{2\cdot r}(g)$.
\end{enumerate}
By Zorn's Lemma for all $r\in\N_{>0}$ there exists an $r$-tiling~\cite[Proposition 5.6.3]{CC}.
\end{def.}
\begin{lemma}[{\cite[Proposition 5.6.4]{CC}}]\label{t:tullio}
 Let $G$ be a finitely generated amenable group and $(S_j)_{j\in\N}$ a F\o lner sequence. Let $r\in\N_{>0}$ and let $T$ be an $r$-tiling for $G$. Set \[T_j:= \{g\in T\mid B_r(g) \subseteq S_j\}\subseteq S_j.\]
Then there exists an $l(T)\in\N$ such that for all $j\geq l(T)$
\[
 \frac{1}{|2\cdot B_{2\cdot r}(e)|} \leq \frac{|T_j|}{|S_j|}\leq  \frac{1}{|B_{r}(e)|}.
\]

\end{lemma}

\begin{rem}\label{r:remark}
 The idea of the following proof of Theorem~\ref{t:sparse} is: Let $(S_j)_{j\in \N}$ be a F\o lner sequence. For each ``ring'' $S_j\setminus S_{j-1}$ we choose a radius $r(j)\in\N_{>0}$ and put part of an $r(j)$-tiling on this ring, see Figure~\ref{f:2}. By letting the sequence $(r(j))_{j\in\N}$ go to infinity, we make sure to get a sparse subset. On the other hand, we let $(r(j))_{j\in\N}$ grow just slowly enough to be sure to get a non-trivial class in $H_0^{\text{uf}}(G;\R)$. We then vary the growth of the sequence $(r(j))_{j\in\N}$ and use Lemma~\ref{l:growth} to get an infinite family of linearly independent classes in $H_0^{\text{uf}}(G;\R)$.  
\end{rem}
\begin{figure}[here]
\begin{center}
  \begin{tikzpicture}[scale = .75]
\foreach \x in {-5,-4,...,5} {
\foreach \y in {-5,-4,...,5}{
\node (a\x\y) at (\x,\y)  {};
\draw[fill, color= lightgray] (a\x\y) circle (3pt);
}
}
\draw[fill] (a00) circle (3pt);
\draw[fill] (a01) circle (3pt);
\draw[fill] (a10) circle (3pt);
\draw[fill] (a-10) circle (3pt);
\draw[fill] (a0-1) circle (3pt);
\draw[fill] (a-40) circle (3pt);
\draw[fill] (a-31) circle (3pt);
\draw[fill] (a-22) circle (3pt);
\draw[fill] (a-13) circle (3pt);
\draw[fill] (a04) circle (3pt);
\draw[fill,yshift=-1cm,xshift=1cm] (-4,0) circle (3pt)
 (-3,1) circle (3pt)
 (-2,2) circle (3pt)
 (-1,3) circle (3pt)
 (0,4) circle (3pt);
\draw[fill,yshift=-1cm,xshift=1cm] (-4,0) circle (3pt)
 (-3,1) circle (3pt)
 (-2,2) circle (3pt)
 (-1,3) circle (3pt)
 (0,4) circle (3pt);
\draw[fill,yshift=-2cm,xshift=2cm] (-4,0) circle (3pt)
 (-3,1) circle (3pt)
 (-2,2) circle (3pt)
 (-1,3) circle (3pt)
 (0,4) circle (3pt);
\draw[fill,yshift=-3cm,xshift=3cm] (-4,0) circle (3pt)
 (-3,1) circle (3pt)
 (-2,2) circle (3pt)
 (-1,3) circle (3pt)
 (0,4) circle (3pt);
\draw[fill,yshift=-4cm,xshift=4cm] (-4,0) circle (3pt)
 (-3,1) circle (3pt)
 (-2,2) circle (3pt)
 (-1,3) circle (3pt)
 (0,4) circle (3pt);
\draw[fill,yshift=-0cm,xshift=0cm] 
(-5,5) circle (3pt)
 (-5,2) circle (3pt)
 (-5,-1) circle (3pt)
 (-5,-4) circle (3pt)
 (-2,5) circle (3pt)
 (1,5) circle (3pt)
 (4,5) circle (3pt)
 (4,2) circle (3pt)
 (4,-1) circle (3pt)
 (4,-4) circle (3pt)
 (-2,-4) circle (3pt)
 (1,-4) circle (3pt)
 (0,4) circle (3pt);
\draw[rotate=45] (-3.25,-3.25)rectangle (3.25,3.25);
\draw[rotate=45] (-1,-1)rectangle (1,1);
 \end{tikzpicture}
\end{center}\caption{Construction of a sparse subset in $\Z^2$ (Remark~\ref{r:remark}).}\label{f:2}
\end{figure}
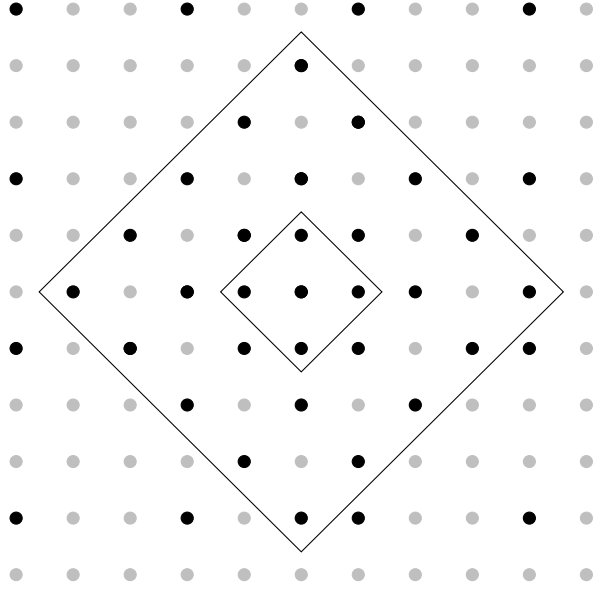
\begin{proof}[Proof of Theorem~\ref{t:sparse}]
 By a result of T. Adachi~\cite{A} there exists a monotonous and exhausting F\o lner sequence of $G$. Hence there is a F\o lner sequence $(S_j)_{j\in \N}$ such that for each $j\in\N$ there exists an $R(j)\in\N_{>0}$ such that
\begin{align*}
 S_{j-1}\subseteq B_{R(j)}(e) \subseteq B_{3\cdot R(j)}(e) \subseteq S_j. \tag{$\ast$}
\end{align*}
Furthermore we can assume $\lim_{j\to \infty} \frac{|S_{j-1}|}{|S_j|} = 0$.

Choose for each $r\in \N_{>0}$ an $r$-tiling~$T^r$ of $G$. For each $j\in \N_{>0}$ and each sequence $c=(c(l))_{l\in\N}$ of positive real numbers converging to 0 define $r(j,c)$ to be the maximal number  $r\in \{1,\dots, R(j)\}$ satisfying

\begin{align*}
 \frac{|B_r(e)|}{4\cdot |B_{2\cdot r}(e)|} &\geq \frac{|S_{j-1}|}{|S_j|}\\
\frac{1}{4\cdot| B_{2\cdot r}(e)|}&\geq \sqrt{c(j)}\\
\intertext{and}
l(T^r)&\leq j.
\end{align*}
If no such $r$ exists, we set $r(j,c)=1$. 

The conditions on $r(j,c)$ will force $(r(j,c))_{j\in\N}$ to grow slowly enough so that the class $\Gamma_c$ we are about to construct will satisfy $c\prec \beta^S_{\Gamma_c}$. 

By our assumption on $(S_j)_{j\in\N}$ we have $\lim_{j\to\infty} r(j,c) = \infty$. Define the subset
\begin{align*}
 \Gamma_c = T_{0}^{r(0,c)} \cup \bigcup_{j\in\N_{>0}} \left(T_j^{r(j,c)} \setminus T_{j-1}^{r(j,c)}\right)\subseteq G.
\end{align*}
For each $r\in\N$ consider $j\in\N$ such that $r(j,c)>2\cdot r$.  By condition $(\ast)$ we know that for each $g\in G\setminus S_{j}$ the ball $B_r(g)$ intersects at most two sets $S_{l}\setminus S_{l-1}$ and $S_{l+1}\setminus S_{l}$. Because of the tiling condition and since $r(j,c)>2\cdot r$ the ball also contains at most one point each from $T^{r(l+1,c)}_{l+1}\setminus T^{r(l+1,c)}_{l}$ and  $T^{r(l,c)}_{l}\setminus T^{r(l,c)}_{l-1}$ hence $|B_r(g)\cap \Gamma_c|\leq 2$, so $\Gamma_c$ is  sparse. We also have for all $j\in\N_{>0}$, by definition of $\Gamma$,
\begin{align*}
 \frac{|\Gamma_c\cap S_j|}{|S_j|}&\geq \frac{|T_j^{r(j,c)}|}{|S_j|}-\frac{|T_{j-1}^{r(j,c)}|}{|S_j|}
 \intertext{by Lemma~\ref{t:tullio}}
 &\geq \frac{1}{2\cdot |B_{2\cdot r(j,c)}(e)|}-\frac{|S_{j-1}|}{|B_{ r(j,c)}(e)|\cdot |S_j|}
 \intertext{and by definition of $r(j,c)$}
&\geq \frac{1}{4\cdot |B_{2\cdot r(j,c)}(e)|}\\
&\geq \sqrt{c(j)}.
\end{align*}
Therefore, $\beta_{\Gamma_c} \succ c$. 

Since $\Gamma_c$ is  sparse, $\lim_{j\to\infty} \beta_{\Gamma_c}(j) = \lim_{j\to\infty} |\Gamma_c\cap S_j|/|S_j| =0$.
Now inductively define a sequence $(\Gamma_n)_{n\in\N}$ of sparse subsets by setting $\Gamma_0 = \Gamma_{\sigma_S}$ and $\Gamma_{n+1} = \Gamma_{\beta^S_{\Gamma_{n}}}$ for all $n\in\N$. Then $\sigma_S\prec \beta_{\Gamma_0}\prec \beta_{\Gamma_1}\prec\cdots$. Hence, by Lemma~\ref{l:growth}, we have found a sequence of linearly independent classes in $\widehat{H}_0^{\uf}(G;\R)$.
\end{proof}

\end{document}